\documentclass[a4paper]{amsart}
\usepackage{graphicx,amsmath,amsfonts,latexsym,amssymb,amsthm,mathrsfs, color}
\usepackage[latin1]{inputenc}
\newcommand{\B}{{\mathbb B}}
\newcommand{\C}{{\mathbb C}}

\newcommand{\N}{{\mathbb N}}

\newcommand{\R}{{\mathbb R}}

\newcommand{\cD}{{\mathcal D}}
\newcommand{\cE}{{\mathcal E}}

\newcommand{\cH}{{\mathcal H}}

\newcommand{\cL}{{\mathcal L}}

\newcommand{\cP}{{\mathcal P}}
\newcommand{\cQ}{{\mathcal Q}}

\newcommand{\cS}{{\mathcal S}}

\newcommand{\ga}{\alpha}
\newcommand{\gb}{\beta}
\renewcommand{\gg}{\gamma}
\newcommand{\gG}{\Gamma}
\newcommand{\gd}{\delta}
\newcommand{\gD}{\Delta}

\newcommand{\gve}{\varepsilon}

\newcommand{\gl}{\lambda}

\newcommand{\go}{\omega}

\newcommand{\gt}{\theta}

\newcommand{\gs}{\sigma}

\newcommand{\Proof}[1]{{\em Proof}. #1~\hfill$\Box$\medskip}
\newcommand{\skp}[1]{\langle#1\rangle}
\renewcommand{\Re}{\mathop{\rm Re}}

\newtheorem{thm}{Theorem}[section]
\newtheorem{proposition}[thm]{Proposition}
\newtheorem{corollary}[thm]{Corollary}
\newtheorem{definition}[thm]{Definition}
\newtheorem{theorem}[thm]{Theorem}
\newtheorem{lemma}[thm]{Lemma}
\newtheorem{remark}[thm]{Remark}

\begin{document}
\title[The Cahn-Hilliard Equation  
on Manifolds with Conical Singularities]%
{The Cahn-Hilliard Equation \\and the Allen-Cahn Equation\\on Manifolds with Conical Singularities}
\author{Nikolaos Roidos}
\author{Elmar Schrohe}
\address{Institut für Analysis, Leibniz Universität Hannover, Welfengarten 1, 
30167 Hannover, Germany}
\email{roidos@math.uni-hannover.de, schrohe@math.uni-hannover.de}

\begin{abstract}We consider the Cahn-Hilliard equation on a manifold with conical 
singularities. We first show the existence of bounded imaginary powers for suitable 
closed extensions of the bilaplacian.
Combining results and methods from singular analysis with a theorem 
of Clément and Li we then prove the short time solvability of 
the Cahn-Hilliard equation in $L_p$-Mellin-Sobolev spaces and obtain the asymptotics
of the solution near the conical points. 

We deduce, in particular,  that regularity is preserved on the smooth part of 
the manifold and singularities remain confined to the conical points.   

We finally show how the Allen-Cahn equation can be treated by simpler considerations.
Again we obtain short time solvability and the behavior near the conical points.
\end{abstract}
\date{}

\maketitle

\section{Introduction}
The Cahn-Hilliard equation is a phase-field or diffuse interface equation 
which is mainly used to model phase separation of a binary mixture, e.g.~a 
two-component alloy, but many other applications are encountered.

In the literature, one finds the equation stated in various forms. 
We shall consider here the version
\begin{eqnarray}\label{CH}\label{e1a}
\partial_t u(t)+\Delta^{2}u(t)+\Delta\big(u(t)-u^{3}(t)\big)&=&0, \quad  t\in(0,T);
\\ 
u(0)&=&u_{0},\label{e1b}
\end{eqnarray}
where $u$ models the concentration difference of the components.
The sets where $u=\pm1$ correspond to domains of pure phases.
The existence of solutions -- even global existence -- is not an issue since 
the work of Elliott and Zheng Songmu \cite{ElliottZhengSongmu86} in 1986
and Caffarelli and Muler \cite{CaffarelliMuler95} in 1995. 
Our main point of interest is to clarify to what extent the singularities
of the underlying space -- here a manifold with conical points -- 
are reflected in a short time solution of the equation.

As usual, we model a manifold with conical singularities by a manifold 
with boundary $\B$ of dimension $n+1$, $n\ge1$, 
endowed with a conically degenerate Riemannian metric. 
On one hand, working on a manifold with boundary simplifies the analysis; on the 
other hand, the degeneracy of the Riemannian metric entails that 
geometric operators such as the Laplacian show the typical degeneracy they
have on spaces with conic points in Euclidean space.

We measure smoothness in terms of weighted Mellin-Sobolev spaces 
$\cH^{s,\gg}_p(\B)$. Here $s$ is a smoothness index, $\gamma$ a weight, and
$1<p<\infty$. 
They coincide with the usual $L^p$-Sobolev spaces away from the singularities.
Close to a conical point, in coordinates $(x,y)$, where $x$ is the distance
to the tip and $y$ a tangential variable, one captures differentiability in 
terms of the operators $x\partial_x$ and $\partial_y$. For $s=0$ we obtain
an $L^p$-space with weight $x^{(\frac{n+1}2-\gg)p-1}$.
It will serve as the base space for our considerations.  

One of the essential points then is to understand the linearized equation, in
particular, the bilaplacian $\Delta^2$ which is the leading order contribution.

As this is a conically degenerate differential operator, 
a first issue is the choice of a suitable closed extension.
Brüning and Seeley \cite{BrueningSeeley88} first noticed that there is 
no canonical choice of a closed extension for such operators.
In general one has a family of closed extensions. 
The domains of the minimal and the maximal extension
differ by a finite-dimensional space of 
functions which are smooth in the tangential variable $y$ and have 
certain asymptotics in $x$ as $x\to0^+$,
see Lesch  \cite{Le} or Schrohe and Seiler \cite{Sh} for more details.

We base our analysis here on that of the Laplacian and choose the 
domain of  $\Delta^2$ accordingly. The closed extensions of the 
Laplacian have been studied in \cite{Sh}. 
Some basic facts are recalled, below.  
We deduce, in particular, that there exist extensions $\underline\Delta$ 
of the Laplacian for which $c-\underline \Delta$ has bounded imaginary
powers on our weighted $L^p$-space for suitably large $c>0$. 
We next show that the corresponding result is true for the bilaplacian
on an appropriately chosen domain which we determine explicitly
in Proposition \ref{dom}. Generically, it is a direct sum of the space
$\cH^{4,4+\gg}_p(\B)$, which is a weighted space of functions
belonging to $H^4_{p,loc}(\B^\circ)$ over the interior $\B^\circ$ of $\B$, 
and a finite-dimensional space of functions with asymptotics near
$x= 0$ as mentioned above. 
Here, the occurring asymptotics types are of the form 
$x^{-q}$ or $x^{-q}\log x$, 
where the exponents $q$ can be determined 
from the spectrum of the Laplace-Beltrami operator $\gD_\partial$
induced by $\gD$ on the cross-section of the cone.
 
Our argument then relies on the notion of maximal regularity:
Let  $X_1\hookrightarrow X_0$ be Banach spaces and let 
$B:\cD(B)=X_1\to X_0$ be a closed  densely defined linear operator. 
Assume that $-B$ generates an analytic semigroup.
Then the operator $B$  is said to have maximal regularity 
for the pair $(X_1,X_0)$ and $1<q<\infty$, if for every $v_0$ in the interpolation
space  $X_q=(X_0,X_1)_{1-1/q,q}$ and every $g\in L^q(0,T;X_0)$
there exists a unique solution 
$v\in L^q(0,T;X_1)\cap W^1_q(0,T;X_0)\cap C([0,T];X_q)$ of the equation 
\begin{eqnarray}\label{e0.2}
\dot v + Bv=g,\ t\in(0,T);\quad v(0)=v_0,
\end{eqnarray}
depending continuously on the data $v_0$ and $g$.

It was proven by Dore and Venni, see Theorem 3.2 in  \cite{DV},
that essentially the existence of bounded imaginary powers for $B$ implies 
maximal regularity, 
see Theorem \ref{DoreVenni} for details. 
Replacing $v$ by $e^{ct}v$, 
it is even sufficient to show that 
$c+B$ has bounded imaginary powers for large positive $c$.

A theorem by Clément and Li shows how maximal regularity can be used to 
establish short time existence of solutions to quasilinear equations of the form  
\begin{eqnarray}\label{e0.1}
\partial_t u(t) + A(u(t))u(t)=f(t,u(t))+g(t), \ t\in(0,T_0);\quad u(0)=u_0
\end{eqnarray}
in 
$X_0$ with domain $\cD(A(u(t)))=X_1$, where $T_0>0$.  
\begin{theorem}\label{CL} {\rm (Clément and Li, \cite{CL}, Theorem 2.1) }
Assume that  there exists an open neighborhood $U$ of 
$u_0$ in $X_q$ such that $A(u_0)$ has maximal regularity for 
$(X_1,X_0)$ and $q$, and that  
\begin{itemize}
\item[(H1)] $A\in C^{1-}(U, \cL(X_1,X_0))$,    
\item[(H2)] $f\in C^{1-,1-}([0,T_0]\times U, X_0)$,
\item[(H3)] $g\in L^q([0,T_0], X_0)$.
\end{itemize}
Then there exists a $T>0$ and a unique 
$u\in L^q(0,T;X_1)\cap W^1_q(0,T;X_0)\cap C([0,T];X_q)$ 
solving the equation \eqref{e0.1} on $]0,T[$.
\end{theorem}

From Theorem \ref{CL} we deduce the short time existence 
of solutions to the Cahn-Hilliard equation.  In our case, the space 
$X_0$ is the weighted $L_p$-space $\cH_p^{0,\gg}(\B)$, 
while $X_1$ is the domain of the bilaplacian. 
The description of $u$ then provides  
information on the regularity of $u$ and its asymptotics near the conical point.
Note that measuring regularity in standard Sobolev spaces is not possible as
our manifold is not even $C^1$-smooth.

We finally turn to the Allen-Cahn equation, a semilinear heat equation of the form
\begin{eqnarray}\label{AC}\label{AC1}
\partial_t u(t)-\Delta u(t)&=&f(u), \quad  t\in(0,T);
\\ 
u(0)&=&u_{0}.\label{AC2}
\end{eqnarray}
Here, $f:\R\to \R$ is a Lipschitz continuous function, which is usually assumed
to be of the form $f=F'$ where $F$ has a double well structure (a fact 
not needed for our arguments). 
For the extensions 
$\underline \gD$ of the Laplacian determined above we immediately  obtain
the existence of a short time solution from Theorem  \ref{CL}. 
Again, the description of the domain provides some asymptotic information.  
 
As both, the Cahn-Hilliard equation and the Allen-Cahn equation are only 
semilinear, we might have relied on a more elementary approach. 
The present setting, however, is rather elegant and allows  us to 
make use of earlier work by Coriasco, Schrohe, and Seiler  \cite{CSS2}.
Moreover, the results for the bilaplacian on 
conic manifolds which we derive here will be useful later on.

This article is structured as follows: In Section 2 we first introduce the
weighted $L^p$ Mellin-Sobolev spaces.  
We next recall the essential facts about domains and extensions 
of the Laplacian on manifolds with straight conical singularities.
Depending on the dimension we then find suitable extensions $\underline \gD$ 
for which $c-\underline\gD$ has bounded imaginary powers for suitably 
large $c>0$.
Section 3 focuses on the description of the domain of the bilaplacian and the 
proof of maximal regularity for the linear part of the equation. 
In Section 4 we apply the above theorem by Clément and Li.
We find that there is a delicate interplay between the choice of the weight 
(and hence the extension) and the conditions of the theorem.
The choices depend on the dimension.
Of course, the two-dimensional case, where the phases may be 
considered as films on a surface with conical singularities, is of greatest practical interest.  
The Allen-Cahn equation is addressed in Section 5.

\section{Notation and Preliminary Results}
\subsection{Bounded imaginary powers} 
Following Amann \cite{Am}, Sections 4.6 and 4.7, 
we give the following two definitions:

\begin{definition}
Let $X$ be a Banach space, $K\geq1$ and $\theta\in[0,\pi[$. 
We denote by  
$\mathcal{P}(K,\theta)$ the class of all closed, densely defined 
linear operators $A$ in $X$ such that 
\[(1+|z|)\,\|(A+z)^{-1}\|\leq K \quad\text{for all } 
z\in S_{\theta}=
\{z\in\mathbb{C}:|\arg z|\leq\theta\}\cup\{0\}\subset\rho{(-A)} .\] 
In particular,  we let $S_{0}=\mathbb{R}^{+}\cup\{0\}$ and write
$\mathcal{P}(\theta)=\bigcup_{K}\mathcal{P}(K,\theta)$.
\end{definition}

\begin{definition}
Let $X$ be a Banach space, $M\geq1$ and $\phi\geq0$. 
We say that a linear operator $A$ in $X$ has \textrm{bounded imaginary powers
with angle} $\phi$ and write $A\in\mathcal{BIP}(M,\phi)$, provided 
$A\in \bigcup_{\theta}\mathcal{P}(\theta)$, the imaginary powers $A^{it}$
are defined for $t\in\R$, and we have the estimate
$$\|A^{it}\|_{\cL(X)}\leq Me^{\phi|t|},\ t\in\mathbb{R}.$$
We let $\mathcal{BIP}(\phi)=\bigcup_{M}\mathcal{BIP}(M,\phi)$.
\end{definition}

The importance of bounded imaginary powers is illustrated 
by the aforementioned result by Dore and Venni \cite{DV}, Theorem 3.2:

\begin{theorem}\label{DoreVenni}Let $X$ be a $\zeta$-convex Banach space and 
$A\in \cP(0)\cap \mathcal{BIP}(\phi)$ for some $0\le \phi<\frac\pi2$.
Then $A$ has maximal regularity for the pair $(\cD(A), X)$.
\end{theorem}
 
\subsection{The Laplacian on Mellin-Sobolev spaces over a manifold with conical singularities}\label{Laplacian}
Let $\B$ be an $n+1$ dimensional smooth compact manifold with 
boundary $\partial\B$. 
We fix a collar neighborhood diffeomorphic to 
$[0,1)\times\partial \B$, where we denote coordinates by $(x,y)$, 
$x\in [0,1)$, $y\in \partial\B$. 

We assume that $\B$ is endowed with a Riemannian metric which, in 
the above neighborhood, takes the degenerate form $g=dx^2+x^2h$, where
$h$ is a Riemannian metric on $\partial \B$.
The associated Laplacian then is a second order cone differential operator.  
It is of the form 
\begin{gather}\label{e111}
\Delta=\frac{1}{x^{2}}\Big((x\partial_{x})^{2}+(n-1)x\partial_{x}+\Delta_{\partial}\Big)
\end{gather}
near the boundary, where $\Delta_{\partial}$ is the Laplacian on 
$\partial\mathbb{B}$ induced by $h$.

By a cut-off function (near $\partial \B$) we mean a smooth non-negative
function $\go$ with $\go\equiv1$ near $\partial \B$  and $\go\equiv0$ 
outside the collar neighborhood  of the boundary. 

\begin{definition}\label{hsgamma}Let $k\in\N_0$, $\gg\in\R$ and 
$1\le p<\infty$. 
By $\cH^{k,\gg}_p(\B)$ we denote the space of all functions $u$ on $\B$
such that for each cut-off function $\go$ we have $(1-\go)u\in H^{k}_p(\B)$
and 
$$x^{\frac{n+1}2-\gamma}(x\partial_x)^j\partial_y^{\alpha}(\go u)(x,y)
\in L^p\left(\frac{dx}xdy\right),\quad j+|\ga|\le k.$$
\end{definition}

There are various ways of extending the definition in order 
to obtain Banach spaces 
$\cH^{s,\gg}_p(\B)$ for all $s\in\R$.
One of the simplest ways, cf.~\cite{CSS1}, is to define the map 
$$\cS_\gamma: C^\infty_c(\R^{n+1}) \to   C^\infty_c(\R^{n+1}),\qquad 
v(t,y) \mapsto e^{(\frac{n+1}2-\gg)t} v(e^{-t},y).$$
Moreover, let $\kappa_{j}:U_{j}\subseteq\partial \B\to\R^n$, $j=1,\ldots,N,$
be a covering of $\partial \B$ by coordinate charts
and $\{\varphi_{j}\}$ 
a subordinate partition of unity. 
Then $\cH^{s,\gg}_p(\B)$ is the space of all distributions 
such that 
    \begin{equation}\label{norm}
	\begin{array}{lcl}
	    \|u\|_{\cH^{s,\gg}_p(\B)} \! & \! = \! & \!
	    \displaystyle
	     \sum_{j=1}^{N}\|\cS_{\gamma}(1\times\kappa_{j})_{*}
	              (\omega\varphi_{j}u)\|_{H^{s}_{p}(\R^{1+n})}+
	             \| (1-\omega)u)\|_{H^{s}_{p}(\B)}
        \end{array}    
    \end{equation}
is defined and finite. Here, $\omega$ is a (fixed) cut-off 
function and $*$ refers to the push-forward of distributions. 
Up to equivalence of norms, this construction is independent of the choice of 
$\go$ and the $\kappa_j$. 
Clearly, $\cH^{0,\gg}_p(\B)$ is a UMD space and 
hence $\zeta$-convex.

\begin{corollary}\label{c1} Let $1\le p<\infty$ and $s>(n+1)/p$. 
Then a function $u$ in 
$\cH^{s,\gg}_p(\B)$ is continuous on $\B^\circ$, and, near $\partial\B$, we have
\begin{eqnarray*}
|u(x,y)|\le c x^{\gg-(n+1)/2} \|u\|_{\cH^{s,\gg}_p(\B)}
\end{eqnarray*}
for a constant $c>0$.
\end{corollary}

\Proof{Continuity on $\B^\circ$ 
follows from the usual Sobolev embedding theorem, 
noting that $\cH^{s,\gg}_p(\B)\hookrightarrow H^s_{p,loc}(\B^\circ)$.
Near the boundary, we deduce from \eqref{norm} and the trace 
theorem that for each $t\in\R$, 
$$e^{((n+1)/2-\gg)t}\|u(e^{-t},\cdot)\|_{B^{s-1/p}_{p,p}(\partial\B)}
\le c \|u\|_{\cH^{s,\gamma}_p(\B)}.
$$
Letting $x=e^{-t}$ we then 
obtain the assertion from the fact that the Besov space 
$B^{s-1/p}_{p,p}(\partial\B)$ embeds into the Sobolev space
$H_p^{s-1/p-\gve}(\partial\B)$ for every $\gve>0$ and the 
Sobolev embedding theorem. }

\subsection{Closed extensions of the Laplace operator}
As pointed out in the introduction,  the choice of a suitable closed extension
of $\Delta$ on $\cH^{0,\gg}_p(\B)$, $1<p<\infty$,  is of central importance. 
For the convenience of the reader we recall here the basic facts following 
\cite{Sh}, where more details can be found.

In the analysis of conically degenerate (pseudo-)differential operators, 
the so-called conormal symbol plays an important role, see e.g. 
Schulze \cite{s94} for an exhaustive treatment.  
Also here, the first step is the analysis of the conormal symbol $\gs_M(\Delta)$ 
of $\Delta$, i.e.  the operator-valued function 
$$\sigma_M(\Delta):\C\to \cL(H^s_p(\partial \B), H^{s-2}_p(\partial \B))\ \ 
\text{given by}\quad
\sigma_M(\Delta)(z)=z^2-(n-1)z+\Delta_\partial.$$
We are interested in the values of $z$ where $\gs_M(\Delta)$ is not invertible.
For this, the precise choice of $s$ and $p$ is not essential.
We denote by
 $0=\lambda_0>\lambda_1>\ldots$ the eigenvalues of $\Delta_\partial$ and by
 $E_0,\,E_1,\ldots$ the corresponding eigenspaces. 
 Moreover, let  $\pi_j\in\cL(L_2(\partial\B))$ be the orthogonal projection 
 onto $E_j$; it  extends to $L^p(\partial\B)$ for $1<p<\infty$: For an 
 $L^2$-orthonormal basis $\{e_{j1},\ldots,e_{jm}\}$ of $E_j$ we let 
 $\pi_j(v) = \sum_{k=1}^m \skp{v,e_{jk}}e_{jk}$.

The {\em non}-bijectivity points of $\sigma_M(\Delta)$ are the 
points $z=q_j^+$ and $z=q_j^-$ with 
  \begin{equation}\label{pjpm}
   q_j^\pm=\mbox{$\frac{n-1}{2}\pm
   \sqrt{\big(\frac{n-1}{2}\big)^2-\lambda_j}$},
   \qquad j\in\N_0.
  \end{equation}
 Note the symmetry $q_j^+=(n-1)-q_j^-$. 
 It is straightforward to see that 
\begin{eqnarray}\label{inverse}
(z^2-(n-1)z+\Delta_\partial)^{-1}=
    \sum_{j=0}^\infty\frac{1}{(z-q_j^+)(z-q_j^-)}\pi_j.
\end{eqnarray}
In fact, this is a pseudodifferential operator which clearly is inverse to 
$\gs_M(\gD)(z)$ on $L^2(\partial\B)$. 
Thus it also is the inverse on $H^s_p(\partial\B)$ for arbitrary $s$ and 
$1<p<\infty$, since the span of the eigenfunctions of $\gD_\partial$ is
dense in these spaces.

Hence, in case $\text{\rm dim}\,\B\not=2$, where the $q^\pm_j$ are all different, 
the inverse to $\gs_M(\Delta)$ has only simple poles in the points $q^\pm_j$.
For $\text{\rm dim}\,\B=2$ the poles at $q_j^\pm$, $j\not=0$, are simple, 
while there is a double pole at $q_0^+=q_0^-=0$.

With $q_j^\pm$, $j\not=0$, we associate the function spaces 
  $$\cE_{q_j^\pm}=\omega\,x^{-q_j^\pm}\otimes E_j=
    \{\omega(x)\,x^{-q_j^\pm}\,e(y): e\in E_j\},\ j\in\N.
   $$
For $j=0$ we let 
\begin{equation}\label{ep0pm}
\cE_{q_0^\pm}=
    \begin{cases}
     \go\otimes E_0+\go\log x\otimes E_0, & 
      \text{\rm dim}\,\B=2\\
     \omega\,x^{q_0^\pm} \otimes E_0,& \text{\rm dim}\,\B\not=2
    \end{cases}.
  \end{equation}
 For later use note that $\gD$ maps the spaces $\cE_{q_j^\pm}$ to 
 $C^\infty_c(\B^\circ)$.  
  
 Furthermore, we introduce the sets $I_\gg$, $\gg\in\R$, by 
  $$I_\gamma=\{q_j^\pm: j\in\N_0\}\cap
    \,\mbox{$]\frac{n+1}{2}-\gamma-2,\frac{n+1}{2}-\gamma[$}.$$

The following is Proposition 5.1 in \cite{Sh}:
 
 \begin{proposition}\label{maxmin}
  The domain of the maximal extension of $\Delta$ in
  $\cH^{0,\gamma}_p(\B)$ is 
   $$\cD(\Delta_{\max})=\cD(\Delta_{\min})\oplus
     \bigoplus_{q_j^\pm\in I_\gamma}\cE_{q_j^\pm}.$$
  In case $q_j^\pm\not=\frac{n+1}{2}-\gamma-2$ for all $j$, the minimal 
  domain is $\cD(\Delta_{\min})=\cH^{2,2+\gamma}_p(\B)$. 
 \end{proposition}

\begin{corollary}\label{domdelta}
The domains of the closed extensions of $\Delta$ are the sets of the form 
$\cD(\Delta_{\min})\oplus\cE,$
where $\cE$ is any subspace of 
$ \mathop{\oplus}_{q_j^\pm\in I_\gamma}\cE_{q_j^\pm}.$
\end{corollary}

\begin{definition}\label{ext}
Given a subspace $\underline\cE_{q_j^\pm}$ of $\cE_{q_j^\pm},$ 
we define the space $\underline{\cE}_{q_j^\pm}^\perp$ as follows: 
 \begin{itemize}
 \item[{\rm i)}] If either $q_j^\pm\not=0$ or 
$\mbox{\rm dim}\,\B\not=2$, there exists a unique subspace 
$\underline{E}_j\subseteq E_j$ such that 
$\underline{\cE}_{q_j^\pm}=\omega\, x^{-q_j^\pm}\otimes \underline{E}_j$. 
Then we set 
$$\underline{\cE}_{q_j^\pm}^\perp
=\omega\,x^{-q_j^\mp}\otimes\underline{E}_j^\perp,$$
where $\underline{E}_j^\perp$ is the orthogonal complement of 
$\underline{E}_j$ in $E_j$ with respect to the $L^2(\partial\B)$-scalar product. 
\item[{\rm ii)}] For $\mbox{\rm dim}\,\B=2$ and $q_0^\pm=0$ define 
$\underline{\cE}_0^\perp=\{0\}$ if $\underline{\cE}_0=\cE_0$, 
$\underline{\cE}_0^\perp=\cE_0$ if $\underline{\cE}_0=\{0\}$, and 
$\underline{\cE}_0^\perp=\underline{\cE}_0$ if 
$\underline{\cE}_0=\go\otimes E_0$. 
\end{itemize}
  Note that $\underline{\cE}_{q_j^\pm}^\perp$ is a subspace of
  $\cE_{q_j^\mp}$. For $\dim \B=2$ we let $\cE_{00}=\go\otimes E_0$.
\end{definition}
 
 We now confine ourselves to extensions $\underline{\Delta}$ with domains 
   $$\cD(\underline{\Delta})=\cD(\Delta_{\min})  
     \oplus\mathop{\mbox{$\bigoplus$}}_{q_j^\pm\in I_\gamma}\underline{\cE}_{q_j^\pm}
     \subseteq \cH^{0,\gg}_p(\B)$$
chosen  according to the following rules: 
   \begin{itemize}
    \item[{\rm(i)}] If $q_j^\pm\in I_\gamma\cap I_{-\gamma}$, then 
     $\underline{\cE}_{q_j^\pm}^\perp=\underline{\cE}_{(n-1)-q_j^\pm}$. 
    \item[{\rm(ii)}] If $\gamma\ge0$ and $q_j^\pm\in I_\gamma\setminus I_{-\gamma}$,
     then $\underline{\cE}_{q_j^\pm}=\cE_{q_j^\pm}$. 
    \item[{\rm(iii)}] If $\gamma\le0$ and 
     $q_j^\pm\in I_{\gamma}\setminus I_{-\gamma}$, 
then $\underline{\cE}_{q_j^\pm}=\{0\}$. 
\footnote{We have corrected in (iii) the order of $ I_{\gamma}$ and 
$I_{-\gamma}$ which was misstated in \cite{Sh}.}     
   \end{itemize}
  In particular, 
  $\cD(\underline{\Delta})=\cD(\Delta_{\max})$ if 
  $\gamma\ge1$ and 
  $\cD(\underline{\Delta})=\cD(\Delta_{\min})$ if 
  $\gamma\le-1$.

 \begin{theorem}\label{elldomain2}Let $\gt\in[0,\pi[$,  and  $\phi>0$.
{\renewcommand{\labelenumi}{{\rm (\alph{enumi})}}
\begin{enumerate}
\item 
For $|\gg|< \dim\B/2$ and $|\gg|< 2$,
let $\underline \Delta$ be an extension with domain chosen as above.
Then $c-\underline \Delta\in \cP(\gt)\cap\mathcal{BIP}(\phi)$ for suitably large $c>0$.  

\item For $\dim\B\ge 4$ and $|\gg|< \dim\B/2$ let $\underline \gD=\gD_{\min}$ for 
$\gg\le0$ and $\underline \gD=\gD_{\max}$ for $\gg>0$. 
Then $c-\underline \Delta\in \cP(\gt)\cap\mathcal{BIP}(\phi)$ for suitably large $c>0$.  
\end{enumerate}
}\end{theorem}

\Proof{(a) For $\dim\B\le 3$ this is  \cite[Theorem 5.7]{Sh} combined with 
\cite[Theorem 4.3]{Sh}.
Inspection shows that  the proof of \cite[Theorem 5.7]{Sh} extends to higher dimensions provided $|\gg|<2$.

(b) is \cite[Theorem 5.6]{Sh} combined with \cite[Theorem 4.3]{Sh}.} 

\begin{remark}\label{r1} {\rm For arbitrary dimension of $\mathbb{B}$, part (a) of Theorem \ref{elldomain2} extends to the case where $|\gamma|< \dim\B/2$ for an extension $\underline \gD$ satisfying the rules (i), (ii) and (iii) above. This follows by iterating the argument given in the proof of \cite[Theorem 5.7]{Sh} and using that the interval $]0,n-1[$ contains none of the $q_{j}^{\pm}$. }
\end{remark}
 
\begin{remark} {\rm An extension $\underline \Delta$  in 
$\mathcal{H}^{0,\gamma}_{p}(\B)$ induces an unbounded operator in 
$L^{q}(0,T;\mathcal{H}^{0,\gamma}_{p}(\B))$, $1<q<\infty$,  
by the relation $(\underline\Delta u)(t)=\underline\Delta(u(t))$.
We denote it again by $\underline\Delta$. }
\end{remark}

\section{The Linearized Problem}
We recall that the gradient associated to the metric $g$, 
$\nabla:C^{\infty}(\mathbb{B}^\circ)\rightarrow 
\gG^\infty(\mathbb{B}^\circ,T\mathbb{B}^\circ)$ is defined by 
\begin{gather*}
\nabla u=\mathrm{grad}\,u=\sum_{ij}g^{ij}\frac{\partial u}{\partial x^{i}}\frac{\partial }{\partial x^{j}},
\end{gather*}
where $(x^1,\ldots, x^{n+1})$ are local coordinates and 
$(g^{ij})=(g_{ij})^{-1}$ is the inverse to the matrix defining $g$ in these
coordinates. 
Near the boundary, $g^{-1}=dx^2 +x^{-2}h^{-1}$ with the notation 
introduced in Section \ref{Laplacian}. 
If $T\B^\circ$ 
is equipped with the Riemannian inner product $(\cdot,\cdot)_g$
given by $g$, 
then 
$$
\Delta u^{3}
= 3u^2\gD u 
-6u \sum g^{ij}\frac{\partial u}{\partial x^i} \frac{\partial u}{\partial x^j}
=3u^{2}\Delta u-6u(\nabla u,\nabla u)_g. 
$$
In coordinates $(x,y^1,\ldots,y^n)$ near the boundary, 
\begin{eqnarray}\label{nabla}
(\nabla u,\nabla v)_g
= \frac1{x^2}((x\partial_x u)( x\partial_x v)
+ \sum_{i,j=1}^n h^{ij}(y)\partial_{y^i}u\,\partial_{y^j}v).
\end{eqnarray}
This allows us to write Equation \eqref{CH} as 
\begin{eqnarray}\label{e3} 
\partial_t u +A(u)u = F(u), \quad u(0)=u_0 
\end{eqnarray}
with
\begin{eqnarray}\label{A}
A(v)u=\Delta^{2}u+\Delta u-3v^{2}\Delta u \,\,\, \mbox{and} \,\,\, 
F(u)=-6u(\nabla u,\nabla u)_g.
\end{eqnarray}
In order to find a suitable domain for the unbounded operator $A(v)$,
we next study the Laplacian. 

\subsection{The choice of an extension of $\gD$}\label{choice}
We proceed to define an extension $\underline\Delta$ of the Laplacian 
on $\cH^{0,\gg}_p(\B)$ satisfying
the assumptions of Theorem \ref{elldomain2}.
Our choice will depend on the dimension. 
We abbreviate
\begin{eqnarray}
\bar\gve = -q_1^->0.
\end{eqnarray}
Note that $\bar\gve$ actually depends on $n$ and the spectrum of $\Delta_\partial$.

\begin{proposition}\label{extdelta} 
The assumptions of Theorem $\ref{elldomain2}$ are fulfilled 
for the choice of extensions outlined in  Sections $\ref{2d}-\ref{4d}$, below. 
\end{proposition}

\subsubsection{The two-dimensional case}\label{2d}
For  $\dim \B = n+1=2$ we pick a weight 
$$-1< \gg<\min\left\{-1+\bar\gve,1\right\}.
$$
This guarantees that $\frac{n+1}2-\gg-2 $ coincides with none of
the $q^\pm_j$ and hence that the minimal domain  
is $\cH^{2,2+\gg}_p(\B)$ by Proposition \ref{maxmin}.
Moreover, $I_\gg\cap I_{-\gg}=\{q_0^\pm\}=\{0\}$.
We choose 
$$\cD(\underline\gD)=\cH^{2,2+\gg}_p(\B)\oplus   \cE_{00},$$
cf.\ Definition \ref{ext}.
By Corollary \ref{c1}, the domain consists of bounded functions only.
Note, moreover, that $ \cE_{00}\subseteq \cH^{\infty,1-\gd}_p(\B)$ for every 
$\gd>0$.

\subsubsection{The three-dimensional case}\label{3d}
For  $\dim\B=3$ we choose  
$$-\frac{1}2<\gg<\min\left\{-\frac{1}2+\bar\gve,\frac{3}{2}\right\}.$$
Then $\frac{n+1}2-\gg-2$ coincides with none of the $q_j^\pm$, and 
the minimal domain 
is $\cH^{2,2+\gg}_p(\B)$  
by Proposition \ref{maxmin}.
The intersection  $I_\gg\cap I_{-\gg}$ equals $\{0,1\}$ for $\gamma<\frac{1}{2}$, and it is empty for $\frac{1}{2}\leq\gamma<\frac{3}{2}$.
According to  Theorem \ref{elldomain2} we choose 
\begin{eqnarray}\label{dome0}
\cD(\underline\gD)= \cH^{2,2+\gg}_p(\B) \oplus \cE_0,
\end{eqnarray}
where $\cE_0$ is the full asymptotics space associated 
with $q^-_0=0$. 

\subsubsection{Higher dimensions}\label{4d}
Next assume $4\le\dim\B$ and choose  
$$\frac{n-3}2<\gg<\min\left\{\frac{n-3}2+\bar\gve,\frac{n+1}{2}\right\}.$$
Then, again $\frac{n+1}2-\gg-2$ does not coincide with any $q_j^\pm$, and 
the minimal domain 
is $\cH^{2,2+\gg}_p(\B)$  
by Proposition \ref{maxmin}.
The intersection  $I_\gg\cap I_{-\gg}$ is empty.
Thus, according to  Theorem \ref{elldomain2} or Remark \ref{r1}, we choose
\begin{eqnarray}
\cD(\underline\gD)= \cH^{2,2+\gg}_p(\B) \oplus \cE_0,
\end{eqnarray}
where again $\cE_0$ is the full asymptotics space of $q^-_0=0$.

\subsection{The domain of $\underline\Delta^2$}
We choose the extension of the bilaplacian induced by our choice of the extension
$\underline \gD$, namely 
\begin{eqnarray*}\label{dom1}
\cD(\underline\gD^2) = \{u\in\cD(\underline\gD): \gD u \in \cD(\underline\gD)\}.
\end{eqnarray*}
According to (2.13) in \cite{Sh}, its conormal symbol is the function 
$$\gs_M(\gD^2) (z) = \gs_M(\gD)(z+2)\gs_M(\gD)(z) .$$
A formula for the inverse follows from \eqref{inverse} and the orthogonality of
the projections $\pi_j$: 
\begin{eqnarray*}
\gs_M(\gD^2)(z)^{-1} &=&   
\sum_{j,k=0}^\infty\frac{1}{(z-q_j^+)(z-q_j^-)(z+2-q_k^+)(z+2-q_k^-)}\pi_j\pi_k\\
&=&\sum_{j=0}^\infty\frac{1}{(z-q_j^+)(z-q_j^-)(z+2-q_j^+)(z+2-q_j^-)}\pi_j.
\end{eqnarray*}
In fact, this is the inverse on $L^2(\partial \B)$.  
As it is a pseudodifferential operator, it extends/restricts to $H^s_p(\partial \B)$ 
for all $1<p<\infty, s\in\R$.
Clearly, we have poles at the points $z=q_j^\pm$ and $z=q_j^\pm-2$.
We denote the collection of all these points by $\cQ$. 
We obtain:
\begin{lemma} {\renewcommand{\labelenumi}{{\rm (\alph{enumi})}}
\begin{enumerate}
\item
If $\dim \B=2$, then we have at least two double poles, namely  at $z=0$
and $z=-2$. 
An additional double pole  occurs 
if $q_j^+-2=q_j^-$ for some $j$. This requires $\gl_j=-1$ for some $j$,
so that this pole will be in  $z=-1$.
\item
If $\dim \B=3$, then a double pole can only occur if 
$q_j^+-2=q_j^-$ for some $j$. As this is precisely the case if $\gl_j=-3/4$,
this pole will be in $z=-1/2$. 
 \item

If $\dim \B=4$, then we have a double pole at $z=0$, since then 
$q^+_0-2=0=q^-_0$.
\item 
For $\dim\B\ge 5$ all poles are simple.
\end{enumerate}}
\end{lemma}

\begin{remark}\rm For the analysis of the bilaplacian 
it is desirable to choose $\gg$ such that the line 
$\{\Re z =\frac{n+1}2-\gg-4\}$ 
does not intersect $\cQ$, for then the minimal domain is 
\begin{eqnarray*}\cD(\gD^2_{\min})= \cH_p^{4,\gg+4}(\B).
\end{eqnarray*}
In case $\cQ$ intersects this line, we have 
\begin{eqnarray*}
\cD(\gD^2_{\min})=\{u\in \bigcap_{\gve>0}\cH^{4,4+\gg-\gve}_p(\B):\gD^2u\in 
\cH^{0,\gg}_p(\B)\}.
\end{eqnarray*}
In particular, 
\begin{eqnarray}\label{mindom}
\cH^{4,4+\gg}_p(\B)\subseteq \cD(\Delta^2_{\min})\subseteq \cH^{4,4+\gg-\gve}_p(\B)\quad 
\text{for all } \gve>0. 
\end{eqnarray}
See \cite[Proposition 2.3]{Sh} for details.
\end{remark}

For a pole $\rho\in \cQ$ of order $k$ we denote by $\tilde \cE_{\rho}$ 
the asymptotics space associated to this pole; it is determined by Equation (2.11) 
in \cite{Sh} and is of the form   
\begin{eqnarray}\label{asy}
\tilde \cE_{\rho} =\textrm{span}\{x^{\rho}\log^l x\,\go(x) e(y):  
l=0,\ldots k-1,e\in \tilde E_{\rho}\},
\end{eqnarray}
where $\tilde E_{\rho}$ is a finite-dimensional subspace of $C^\infty(\partial\B)$
consisting of eigenfunctions of $\gD_\partial$.
Note that in our case $k$ can only take the values $1$ and $2$. 
 
Now we know on one hand that, for $0\not=e\in C^\infty(\partial \B)$,   
$$x^{-\rho}\log^lx\,\go(x)e\in \cH_p^{s,\gg}(\B) \text{ if and only if }
\Re\, \rho<\frac{n+1}2-\gg;$$ 
on the other hand, for a pole in $\rho\in\cQ$ of order $k$, $l<k$, and $e\in \tilde E_q$, 
$$\gD(x^{-\rho}\log^lx\,\go(x)e) \in C^\infty_c(\B^\circ)\subseteq \cH_p^{\infty,\infty}(\B).$$

\begin{proposition}\label{dom}We define the interval 
$$J=J(n,\gg) = \ ]\frac{n+1}2-\gg-4, \frac{n+1}2-\gg-2[.$$
With the choices made in Sections $\ref{2d}$--$\ref{4d}$ we have 
{\renewcommand{\labelenumi}{{\rm (\alph{enumi})}}
\begin{enumerate}
\item For $\dim\B=2$ and the extension $\underline \gD$ in $\ref{2d}$ 
we have 
\begin{eqnarray*}\label{Delta2}
\cD(\underline \gD^2)= \cD(\Delta^2_{\min})\oplus 
\bigoplus_{\rho\in J}\tilde \cE_\rho\oplus \mathcal{E}_{00}.\\
\end{eqnarray*}

\item For $\dim\B\geq3$ and the extension $\underline \gD$ in $\ref{3d}$ or $\ref{4d}$ we have
\begin{eqnarray*}\label{dom3}
\cD(\underline \gD^2)=  \cD(\Delta^2_{\min})\oplus  
\bigoplus_{\rho\in\, J}\tilde \cE_\rho \oplus \cE_0.
\end{eqnarray*}

\end{enumerate}} 
\end{proposition}

\begin{corollary}\label{poles}\rm
We can now describe the domain of the bilaplacian explicitly, using \eqref{asy}. 
Let $J$ be the interval introduced in Proposition \ref{dom}.
{\renewcommand{\labelenumi}{{\rm (\alph{enumi})}}
\begin{enumerate}
\item For $\dim \B=2$ and the extension in \ref{dom}(a),  where 
$-1<\gg<\min\{-1+\overline\gve,1\}$, the interval $J=\ ]-3-\gg, -1-\gg[$ will contain the 
double pole in $z=-2$, but not that in $z=0$. Depending on $\gg$ and $\overline\gve$,
it might contain the possible double pole in $z=-1$.

\item For $\dim \B=3$ and the extension in \ref{dom}(b), the interval $J$
might contain the only possible double pole at $z=-1/2$; it will if 
$\overline\gve<1/2$. 

\item For $\dim\B=4$ and the extension in \ref{dom}(b), there is a double pole in 
$z=0$ which is not contained in $J=\ ]-\gg-2,-\gg[$.

\item For $\dim\B>4$ no double poles arise. 
\end{enumerate}
}
\end{corollary}

\subsection{Embedding the interpolation space $X_{q}$}\label{Ep}
We shall apply the theorem of Clément and Li with the choices 
$X_0=\cH^{0,\gg}_p(\B)$ and $X_1=\cD(\underline \gD^2)$ for 
$2<q<\infty$.
We next  look for a suitable embedding of the interpolation space $X_q$. 
For arbitrary $\eta$ with $1/2<\eta<1-1/q$ we have
\begin{gather}\label{e32}
X_{q}:=(X_{0},X_{1})_{1-\frac{1}{q},q}
\hookrightarrow[X_{0},X_{1}]_{\eta}
=[\cH^{0,\gg}_p(\B), \cD(\underline{\Delta}^2)]_{\eta}=
[\cD((c-\underline{\Delta})^0),\cD((c-\underline{\Delta})^2) ]_\eta,
\end{gather}
for suitably large $c>0$. 
Since $c-\underline{\Delta}\in\mathcal{BIP}(\phi)$ for any $\phi>0$ and 
sufficiently large $c>0$, we apply  (I.2.9.8) in \cite{Am} and obtain
\begin{gather}\label{e62}
[\cD((c-\underline{\Delta})^{0}),\cD((c-\underline{\Delta})^{2})]_{\eta}
\hookrightarrow\cD((c-\underline{\Delta})^{(1-\eta)0+2\eta})
=\cD((c-\underline{\Delta})^{2\eta}).
\end{gather} 
As $\eta>1/2$, we have $2\eta = 1+\vartheta$ for some $\vartheta>0$. 
We apply once more (I.2.9.8) in \cite{Am} and use the fact that 
$\cD(\underline\gD)\subseteq \cH^{2,\gg+\gve_0}_p(\B)$ and  
$\cD(\underline\gD^2)\subseteq \cH^{4,\gg+\gve_1}_p(\B)$ for suitable 
$\gve_0,\gve_1>0$ and $0<\vartheta'<\vartheta$:
\begin{eqnarray}
\lefteqn{\cD((c-\underline\gD)^{2\eta})=
\cD((c-\underline{\Delta})^{1+\vartheta})}\nonumber\\
&=&[\cD(c-\underline\gD), \cD((c-\underline\gD)^2)]_{\vartheta}\hookrightarrow
[\cH^{2,\gg+\gve_0}_p(\B), \cH^{4,\gg+\gve_1}_p(\B)]_{\vartheta}\label{embed}\\
&\hookrightarrow&(\cH^{2,\gg+\gve_0}_p(\B), \cH^{4,\gg+\gve_1}_p(\B))_{\vartheta',p}.
\nonumber
\end{eqnarray}
Next we use Lemma 5.4 in \cite{CSS1} to conclude that, for arbitrary 
$\gd_0, \gd_1>0$, we have   
\begin{eqnarray*}
\lefteqn{(\cH^{2,\gg+\gve_0}_p(\B), \cH^{4,\gg+\gve_1}_p(\B))_{\vartheta',p}}\\
&\hookrightarrow &
\cH_p^{4\vartheta'+2(1-\vartheta')-\gd_0, \gg+\vartheta'\gve_1+(1-\vartheta')\gve_0-\gd_1}(\B)
=\cH_p^{2+2\vartheta'-\gd_0, \gg+\vartheta'\gve_1+(1-\vartheta')\gve_0-\gd_1}(\B).
\end{eqnarray*}
Summing up, we see that 
\begin{eqnarray*}X_q \hookrightarrow \cD(\underline\gD)\cap
\cH_p^{2+2\vartheta-\gd_0, \gg+\vartheta\gve_1+(1-\vartheta)\gve_0-\gd_1}(\B)
\end{eqnarray*}
for every $\vartheta$ with $0<\vartheta<1-2/q$.

For the extensions in Proposition \ref{extdelta} we  have 
$\cD(\underline \gD) \subseteq L^\infty(\B)$ and hence $X_q\subseteq L^\infty(\B)$. 
By \eqref{A},
$$\cD(A(v)) = \cD(\underline \gD^2),\quad v\in X_q.$$ 

\subsection{Bounded imaginary powers}

The following observation might be well-known. As we did
not find a reference, we include a proof: 

\begin{lemma}\label{l1}
Let $E$ be a Banach space and $A\in\mathcal{P}(\theta)$ with $\theta\geq\pi/2$. 
Then $A^{2}\in\mathcal{P}(\tilde \theta)$ for $\tilde\theta= 2\theta-\pi$ 
and $(A^{2})^{z}=A^{2z}$ for $z\in \C$.
\end{lemma}

\begin{proof}
In view of the fact that $A^{-2z}$ and $(A^2)^{-z}$ are holomorphic operator
families for $\Re(z)>0$ we can confine ourselves to the case 
$0<\Re(z)<\frac{1}{2}$.

Let  $A\in\mathcal{P}(K,\theta)$ for $K\ge1$. 
The resolvent formula implies that
\begin{gather}\label{e444}
(A^{2}+\gl )^{-1}=(A-i\sqrt{\gl })^{-1}(A+i\sqrt{\gl })^{-1}=
\frac{1}{2i\sqrt{\gl }}\left((A-i\sqrt{\gl })^{-1}-(A+i\sqrt{\gl })^{-1}\right).
\end{gather}
We note that $\arg(\pm i\sqrt\gl)=\frac12\arg\gl\pm\frac12\pi$. 
Thus, for $\gl \in S_{\tilde\gt}$ with $|\gl|$ away of zero,
\begin{gather*}
(1+|\gl| )\|(A^{2}+\gl )^{-1}\|
\leq\frac{1+|\gl| }{2|\sqrt{\gl }|}\left(\|(A-i\sqrt{\gl })^{-1}\|+\|(A+i\sqrt{\gl })^{-1}\|\right)
\leq \tilde{K},
\end{gather*} 
for some $\tilde{K}>0$,
and hence $A^{2}\in\mathcal{P}(K',\tilde\theta)$ for some $K'\geq1$. 
Following Amann, cf.\ (III.4.6.9) in \cite{Am},  we let 
\begin{gather*}
(A^{2})^{-z}=
\frac{\sin\pi z}{\pi}\int_{0}^{+\infty}u^{-z}(A^{2}+u)^{-1}du , \quad
0<\Re(z)<\frac{1}{2}.
\end{gather*}
By (\ref{e444}) we find that
\begin{eqnarray*}
\lefteqn{(A^{2})^{-z}
=\frac{\sin\pi z}{\pi}
\int_{0}^{+\infty}\frac{u^{-z}}{2i\sqrt{u}}
\left((A-i\sqrt{u})^{-1}-(A+i\sqrt{u})^{-1}\right)du} \\
&=&\frac{\sin\pi z}{\pi}
\int_{0}^{-i\infty}(-\lambda^{2})^{-z}(A+\lambda)^{-1}d\lambda
+\frac{\sin\pi z}{\pi}
\int_{0}^{+i\infty}(-\lambda^{2})^{-z}(A+\lambda)^{-1}d\lambda \\
&=&-\frac{e^{i\pi z}-e^{-i\pi z}}{2\pi i}
\int_{-i\infty}^{0}(e^{i\pi}\lambda^{2})^{-z}(A+\lambda)^{-1}d\lambda
+\frac{e^{i\pi z}-e^{-i\pi z}}{2\pi i}
\int_{0}^{+i\infty}(e^{-i\pi}\lambda^{2})^{-z}(A+\lambda)^{-1}d\lambda\\
&=&\frac{1}{2\pi i}
\int_{-i\infty}^{0}(e^{i\pi}\lambda)^{-2z}(A+\lambda)^{-1}d\lambda
+\frac{1}{2\pi i}
\int_{0}^{+i\infty}(e^{-i\pi}\lambda)^{-2z}(A+\lambda)^{-1}d\lambda
-\frac{1}{2\pi i}
\int_{-i\infty}^{+i\infty}\lambda^{-2z}(A+\lambda)^{-1}d\lambda\\
&=&\frac{1}{2\pi i}
\int_{-i\infty}^{+i\infty}(-\lambda)^{-2z}(A+\lambda)^{-1}d\lambda=A^{-2z},
\end{eqnarray*}
where we have used the fact that 
\begin{gather*}
\int_{-i\infty}^{+i\infty}\lambda^{-2z}(A+\lambda)^{-1}d\lambda=0,
\end{gather*}
since $\lambda^{-2z}$ is holomorphic for $\Re(\lambda)>0$.
\end{proof}\medskip

In the following proposition, $\underline{\Delta}$ denotes 
the dilation invariant extension of the Laplacian  defined in Section \ref{choice}
and $A(u_0)$ is the operator defined in \eqref{A} with the choice 
$\underline\gD$ for the Laplacian.

\begin{proposition}\label{p1}
For every choice of  
$u_0\in L^{\infty}(\mathbb{B})$, $\phi>0$, and $\gt\in[0,\pi[$,
the operator $A(u_0)+c_0I$, 
considered as an unbounded operator in 
$\mathcal{H}^{0,\gamma}_{p}(\mathbb{B})$ with domain 
$\cD(\underline{\Delta}^{2})$ belongs to $\cP(\gt)\cap\mathcal{BIP}(\phi)$ 
for all sufficiently large $c_0>0$.
\end{proposition}

\begin{proof} By possibly increasing $\gt$ we may  assume that 
$\max\{\pi-\gt,\phi\}=\phi$. 
Theorem \ref{elldomain2} asserts that  
$A=c-\underline{\Delta}$ belongs to 
$\mathcal{P}(K,(\theta+\pi)/2)\cap\mathcal{BIP}(\phi/2)$ with suitable $K$, 
provided $c$ is large.
Now Lemma \ref{l1} implies that 
$
A^{2}\in\mathcal{P}(\theta)$ with
$(A^{2})^{z}=A^{2z}$ for $\Re(z)<0$
and hence  that 
$A^{2}\in\mathcal{BIP}(\phi)$.

Moreover, 
$A^{2}+\mu\in\mathcal{P}(\theta)$ for $\mu\geq0$. 
By Corollary III.4.8.6 in \cite{Am}, 
$A^{2}+\mu\in\mathcal{BIP}(\max\{\pi-\gt,\phi\})=
\mathcal{BIP}(\phi)$.
In order to obtain $\mathcal{BIP}$ for  $A(u_0)$, 
we apply a perturbation result, namely 
Theorem III.4.8.5 in \cite{Am}  for the perturbation 
$B=2c\underline{\Delta}$.
In accordance with the notation used there, we denote by   
$\Gamma(k,\psi)$  the negatively oriented boundary of 
\begin{gather*}
\{|\arg(z)|\leq \psi_{k}\}\cup\{|z|\leq 1/2k\}, \,\,\, 
\mbox{where} \,\,\, \psi_{k}=
\min\left\{\frac{\pi+\psi}{2},\mathrm{arcsin}\frac{1}{2k}\right\}.
\end{gather*}
We will have $(A^2+\mu)+B=\underline\gD^2+c^2+\mu
\in\mathcal{BIP}(\phi)$, if we can show that for suitable  
$0<\gb<1$ and $K_1\ge1$
\begin{enumerate}\renewcommand{\labelenumi}{({\roman{enumi})}}
\item $\|B(A^2+\mu+\gl)^{-1}\|\le \gb$ for all $\gl\in\gG\cup S_{\gt}$, 
where $\gG =\gG((1-\gb)^{-1}K_1,\gt)$ 
\item $(A^2+\mu+\gl)^{-1}B(A^2+\mu+\gl)^{-1}\in 
L^1(\gG,d\gl, \cL(\cH^{0,\gg}_p(\B)))$.
\end{enumerate}

Concerning (i): 
\begin{eqnarray*}
\lefteqn{\|B(A^2+\mu+\lambda)^{-1}\|
=2c
\|\underline{\Delta}((c-\underline{\Delta})^{2}+\mu+\lambda)^{-1}\|}\\
&=&2c
\|\underline{\Delta}(-\underline{\Delta}+c-i\sqrt{\mu+\lambda})^{-1}
  (-\underline{\Delta}+c+i\sqrt{\mu+\lambda})^{-1}\|\\
&\leq&
2c\|(-\underline{\Delta}+c-i\sqrt{\mu+\lambda})^{-1}\|
\|(-\underline{\Delta}+c+i\sqrt{\mu+\lambda}-(c+i\sqrt{\mu+\lambda}))
(-\underline{\Delta}+c+i\sqrt{\mu+\lambda})^{-1}\|\\
&\leq&
\frac{2cK}{1+|\sqrt{\mu+\lambda}|}
\|I-(c+i\sqrt{\mu+\lambda})(-\underline{\Delta}+c+i\sqrt{\mu+\lambda})^{-1}\|\\
&\leq&
\frac{2cK}{1+|\sqrt{\mu+\lambda}|}
\left(1+\frac{K|c+i\sqrt{\mu+\lambda}|}{1+|\sqrt{\mu+\lambda}|}\right)
\end{eqnarray*}
The last expression can be estimated by $\gb$ provided $\mu$ is taken  
sufficiently large.

Concerning (ii):  
\begin{eqnarray*}
\lefteqn{\|(A^2+\mu+\lambda)^{-1}B(A^2+\mu+\lambda)^{-1}\|
=\|((-\underline{\Delta}+c)^{2}+\mu+\lambda)^{-1}2c\underline{\Delta}
((-\underline{\Delta}+c)^{2}+\mu+\lambda)^{-1}\|}\\
&=&
2c\|(-\underline{\Delta}+c-i\sqrt{\mu+\lambda})^{-1}
(-\underline{\Delta}+c+i\sqrt{\mu+\lambda})^{-1}
\Big(-\underline{\Delta}+c-i\sqrt{\mu+\lambda}-(c-i\sqrt{\mu+\lambda})\Big)\\
&&(-\underline{\Delta}+c-i\sqrt{\mu+\lambda})^{-1}
(-\underline{\Delta}+c+i\sqrt{\mu+\lambda})^{-1}\|\\
&\leq& 
2c\Big(\frac{K}{1+|\sqrt{\mu+\lambda}|}\Big)^{3}
\Big(1+\frac{K|c-i\sqrt{\mu+\lambda}|}{1+|\sqrt{\mu+\lambda}|}\Big)
=O(|\mu+\lambda|^{-\frac{3}{2}})
\end{eqnarray*}
so that also (ii) holds.

We write $\tilde c=c^2+\mu$  for $\mu$ as above  
and apply once more Theorem III.4.8.5 in \cite{Am}, now with 
$\underline{\Delta}^{2}+\tilde{c}$ in the role of $A$ and 
$h\underline{\Delta}$ in the role of $B$ for an arbitrary $h\in L^\infty(\B)$.
Obviously, 
$\cD(h\underline\gD)\supseteq\cD(\underline\gD^2+\tilde c)$. 
From (\ref{e444})  we see that 
$\underline{\Delta}^{2}+\tilde{c}\in \mathcal{P}(\tilde K,\theta)$ 
for some $\tilde K\geq 1$. 
We now have to check the analogs of conditions (i) and (ii) above.
Similarly as before, we note concerning (i) that, by possibly increasing $\mu$,
\begin{eqnarray*}\nonumber
\lefteqn{
\|h\underline{\Delta}\left(\underline{\Delta}^{2}+\tilde{c}+\lambda\right)^{-1}\|
\leq
\|h\|_{\infty}\|\underline{\Delta}
\left(-\underline{\Delta}+i\sqrt{\tilde{c}+\lambda}\right)^{-1}\left(-\underline{\Delta}-i\sqrt{\tilde{c}+\lambda}\right)^{-1}\|}\\\label{e21}
&\leq&
\|h\|_{\infty}
\left\|\Big(I-i\sqrt{\tilde{c}+\lambda}\left(-\underline{\Delta}+i\sqrt{\tilde{c}+\lambda}
\right)^{-1}\Big)\left(-\underline{\Delta}-i\sqrt{\tilde{c}+\lambda}
\right)^{-1}\right\|\\
&\leq&
\|h\|_{\infty}
\left(1+\frac{K|\sqrt{\tilde{c}+\lambda}|}%
{1+|c-i\sqrt{\tilde{c}+\lambda}|}\right)
\frac{K}{1+|c+i\sqrt{\tilde{c}+\lambda}|}
\leq\beta,
\end{eqnarray*}
for 
$\lambda$ in $\Gamma((1-\beta)^{-1}K'',\theta)\cup S_{\theta}$, 
with $K''\geq \tilde K$ sufficiently large. 
Moreover, concerning (ii), 
\begin{eqnarray*}
\lefteqn{
\|(\underline{\Delta}^{2}+\tilde{c}+\lambda)^{-1}
h\underline{\Delta}(\underline{\Delta}^2+\tilde{c}+\lambda)^{-1}\|}\\
&\leq&
\|(\underline{\Delta}^{2}+\tilde{c}+\lambda)^{-1}\|
\|h\|_{\infty}\|\underline{\Delta}(\underline{\Delta}^{2}+\tilde{c}+\lambda)^{-1}\|\\
&\leq&
\|h\|_{\infty}
\|(-\underline{\Delta}-i\sqrt{\tilde{c}+\lambda})^{-1}
(-\underline{\Delta}+i\sqrt{\tilde{c}+\lambda})^{-1}\|\\
&&\left\|
\big(I+i\sqrt{\tilde{c}+\lambda}
(-\underline{\Delta}-i\sqrt{\tilde{c}+\lambda})^{-1}\big)
(-\underline{\Delta}+i\sqrt{\tilde{c}+\lambda})^{-1}
\right\|\\
&\leq&
\|h\|_{\infty}
\left(\frac{K}{1+|c-i\sqrt{\tilde{c}+\lambda}|}\right)^2
\left(1+\frac{K|\sqrt{\tilde{c}+\lambda}|}%
{1+|c+i\sqrt{\tilde{c}+\lambda}|}\right)
\frac{K}{1+|c+i\sqrt{\tilde{c}+\lambda}|}
=O(\lambda^{-\frac{3}{2}}),
\end{eqnarray*}
so that the operator is $L^1\left(\Gamma((1-\beta)^{-1}K'',\theta); 
\mathcal{L}(\mathcal{H}^{0,\gamma}_{p}(\mathbb{B}))\right)$.
Choosing $h=1-3u_{0}^{2}$ we get the result.
\end{proof}

\section{The Nonlinear Equation in $\cH^{0,\gg}_p(\B)$}

\subsection{The two-dimensional case.} Let $\dim\B=2$ and $\underline\gD^2$ 
the extension of the bilaplacian determined in Proposition \ref{dom}(a).

\begin{lemma}\label{4.1} Let $(x,y)$ be local coordinates near $\partial \B$.  
For $u\in X_q$ with $q>2$ we have $x\partial_x u$ and $\partial_{y}u$ in 
$\cH^{1+\gve,2+\gg}_p(\B)$ for sufficiently small $\gve>0$. They are therefore 
$L^\infty$-functions whenever $p\geq2$. 
\end{lemma}

\begin{proof}We recall from Section \ref{Ep} that $X_q$ embeds into 
the interpolation space $[\cD(\underline\gD),\cD((\underline\gD)^2)]_\vartheta$ 
for some $\vartheta>0$. 
Since the functions in $E_0$ are locally constant, the functions in
$\cE_{00}$ are also locally constant near the boundary. 
Taking any derivative will result in a function which equals zero near the boundary.  
Hence, the operators $x\partial_x$ and $\partial_y$ map 
$\cD(\underline\gD)$ to $\cH_p^{1,2+\gg}(\B)$ and as they preserve the asymptotics of the functions in $\tilde{\mathcal{E}}_{\rho}$, with $\rho\in J$, they map $\cD(\underline\gD^2)$ to $\cH^{3,2+\gg+\varepsilon}_p(\B)$ for some $\varepsilon>0$. Thus, they map $X_q$ to $[\cH_p^{1,2+\gg}(\B),\cH_p^{3,2+\gg+\varepsilon}(\B)]_\vartheta$ for some $\vartheta>0$. The latter space embeds to $\cH_p^{1+\delta,2+\gg+\delta}(\B)$ for some $\delta>0$, which is a subset of $L^\infty(\B)$.
\end{proof}

\begin{corollary}\label{c11}
We infer from  \eqref{nabla} that for $u,v\in X_q$ supported near $\partial \B$, 
\begin{eqnarray*}
\lefteqn{\|(\nabla u,\nabla v)_g\|_{\cH^{0,\gg}_p(\B)}}\\
&\le &c_1
\max\{\|x\partial_x u\|_{L^\infty},\|\partial_{y}u\|_{L^\infty(\B)}\} 
\\
&&\ \ \times \max\{\|x^{-2}(x\partial_x v)\|_{\cH^{0,\gg}_p(\B)},
\|x^{-2}\partial_{y}v\|_{\cH^{0,\gg}_p(\B)}\}\le c_2\|u\|_{X_q}\|v\|_{X_q}
\end{eqnarray*}
with  suitable constants $c_1,c_2$.
\end{corollary}

\begin{theorem}\label{one}
Let $\dim\B=2$, $q>2$ and $p\geq 2$.
Given any $u_0\in X_q$, there exists a $T>0$ and a unique solution
$$u\in L^q(0,T; \cD(\underline\gD^2))\cap W^{1}_q(0,T; \cH^{0,\gg}_p(\B))
\cap C([0,T],X_q)$$
solving Equation \eqref{e1a} on $]0,T[$ with initial condition \eqref{e1b}. 
\end{theorem}

\begin{proof}
Given $\phi>0$ we know from Proposition \ref{p1} that $A(u_0)+c_0I$ has 
$\mathcal{BIP}(\phi)$ provided $c_0>0$ is large. Hence we have maximal 
regularity by Dore and Venni's theorem.
Next let us check conditions (H1) and (H2) in Clément and Li's theorem;
note that (H3) is not required for this particular equation.
Let $U $ be a bounded neighborhood of $u_0$ in  $X_q$. 
We noted in Section \ref{choice} that $U$ consists of bounded functions.

Concerning (H1): Let $u_1, u_2\in U$. Then
\begin{eqnarray*}
\lefteqn{\|A(u_{1})-A(u_{2})\|_{\mathcal{L}(X_{1},X_{0})}
=3\|(u_{1}^{2}-u_{2}^{2})\underline{\Delta}\|_{\mathcal{L}(X_{1},X_{0})}
\leq c\|(u_{1}^{2}-u_{2}^{2})I\|_{\mathcal{L}(\cD(\underline{\Delta}),X_{0})}}
\\
&\leq&c_{1}\|(u_{1}-u_{2})(u_{1}+u_{2})\|_\infty
\leq c_{2}(\|u_1\|_{X_q}+\|u_2\|_{X_q})\|u_1-u_2\|_{X_q}\le c_3\|u_1-u_2\|_{X_q} 
\end{eqnarray*}
for suitable constants $c_{1}$, $c_{2}$ and $c_{3}$, where the last inequality
is a consequence of the boundedness of $U$. 

Concerning (H2), we argue that 
\begin{eqnarray*}
\lefteqn{\|F(u_{1})-F(u_{2})\|_{X_{0}} 
= \|6u_1(\nabla u_{1},\nabla u_{1})_g
   -6u_{2}(\nabla u_{2},\nabla u_{2})_g\|_{X_{0}}
 }\\
 &\le&6\|u_{1}(\nabla u_{1},\nabla u_{1})_g
 -u_{2}(\nabla u_{1},\nabla u_{1})_g\|_{X_0}
 +6\|u_{2}(\nabla u_{1},\nabla u_{1})_g 
 -u_{2}(\nabla u_{1},\nabla u_{2})_g\|_{X_0}\\
&& +6\|u_{2}(\nabla u_{1},\nabla u_{2})_g
 -u_{2}(\nabla u_{2},\nabla u_{2})_g\|_{X_{0}} \\
 &\leq& 6(\|u_{1}-u_{2}\|_\infty\|(\nabla u_{1},\nabla u_{1})_g\|_{X_{0}}
 +6\|u_{2}\|_\infty \|(\nabla u_{1},\nabla (u_{1}-u_{2}))_g\|_{X_{0}}\\
 &&+6\|u_{2}\|_\infty \|(\nabla u_{2},\nabla (u_{1}-u_{2}))_g\|_{X_{0}}.
\end{eqnarray*}

According to Corollary \ref{c11}, we can estimate the right hand side by 
$c \|u_1-u_2\|_{X_q}(\|u_1\|_{X_q}+\|u_2\|_{X_q})^2$, which is bounded in view of 
the boundedness of $U$.
\end{proof}

\subsection{The higher-dimensional case.} Let $\dim\B\geq3$, $q>2$, 
$p\ge n+1$ and  $\underline\gD^2$ the extension determined in 
Proposition \ref{dom}(b).
As $z=0$ is a simple pole for the inverted Mellin symbol of the Laplacian, $\cE_0$ 
consists of bounded functions which are locally constant near the boundary. 
Hence $\cD(\underline\gD)$ embeds into $L^\infty(\B)$ and 
so does  $\cD(\underline\gD^2)\subseteq  \cD(\underline\gD) $.
By interpolation $X_q\hookrightarrow L^\infty(\B)$. We have the following analog of Lemma \ref{4.1}:

\begin{lemma}\label{grad}
Let $(x,y^1,...,y^n)$ be local coordinates near $\partial\B$, $q>2$
and $p\geq n+1$.  
Then $x\partial_x$ and $\partial_{y^j}$, $j\in\{1,...,n\}$, are bounded maps from  
$X_q$ to $\cH^{1+\gve,2+\gg}_p(\B)\hookrightarrow L^\infty(\B)$ for sufficiently small 
$\gve>0$.
\end{lemma}

Then, similarly to Theorem \ref{one}, we prove

\begin{theorem}\label{3dim}
Let $\dim\B\geq3$, $q>2$, $p\geq n+1$ and $\underline\gD^2$ the extension 
determined in Proposition {\rm \ref{dom}(b)}.
Given any $u_0\in X_q$, there exists a $T>0$ and a 
unique solution
$$u\in L^q(0,T; \cD(\underline\gD^2))\cap W^{1}_q(0,T; \cH^{0,\gg}_p(\B))
\cap C([0,T],X_q)$$
solving Equation \eqref{e1a} on $]0,T[$ with initial condition \eqref{e1b}. 
\end{theorem}

\begin{proof}
The fact that $u_0$ is in $L^\infty$ yields maximal regularity for 
$A(u_0)+\tilde c$ by Proposition \ref{p1} and Dore and Venni's 
theorem for large  $\tilde c$.
So we only have to check conditions (H1) and (H2) in Clément and Li's 
theorem.
Let $U $ be a bounded neighborhood of $u_0$ in  $X_q$ and 
$\tilde x$ a function  which equals $x$ near $\partial\B$, is strictly positive on 
$\B^\circ$ and is $\equiv 1$ outside a neighborhood of $\partial \B$.

Concerning (H1): Let $u_1, u_2\in U$. Then
\begin{eqnarray*}
\lefteqn{\|A(u_{1})-A(u_{2})\|_{\mathcal{L}(X_{1},X_{0})}
=3\|(u_{1}^{2}-u_{2}^{2})\underline{\Delta}\|_{\mathcal{L}(X_{1},X_{0})}
\leq c\|(u_{1}^{2}-u_{2}^{2})I\|_{\mathcal{L}(\cD(\underline{\Delta}),X_{0})}}
\\
&\le& 
c\|(u_1+u_2)(u_1-u_2)I\|_{\mathcal{L}(\cD(\underline{\Delta}),\cH^{0,\gg}_p(\B))}
\le c_1 \|u_{1}+u_{2}\|_\infty \|u_{1}-u_{2}\|_\infty\\
&\leq&
 c_{2}(\|u_1\|_{X_q}+\|u_2\|_{X_q})\|u_1-u_2\|_{X_q}
\le c_3\|u_1-u_2\|_{X_q} 
\end{eqnarray*}
for suitable constants $c, c_{1}, c_{2},$ and $c_{3}$, 
where the last inequality is a consequence of the boundedness of $U$. 

Concerning (H2), we first deduce from Lemma \ref{grad} and Corollary \ref{c1} that 
for $u,v\in X_q$, we have near $\partial \B$ 
\begin{gather*}
|x^{2}(\nabla u,\nabla  v)_g|=|(x\partial_x u)(x\partial_x v)
+\sum h^{ij}(\partial_{y^i} u)(\partial_{y^j}v)|
\leq c_{4} x^{2(2+\gamma-\frac{n+1}{2})}\|u\|_{X_q}\|v\|_{X_q}
\end{gather*}
for some constant $c_{4}$. Noting that $\gg>\frac{n-3}2$ we can estimate
\begin{eqnarray*}
\lefteqn{\|F(u_{1})-F(u_{2})\|_{X_{0}} 
= \|6u_1(\nabla u_{1},\nabla u_{1})_g
   -6u_{2}(\nabla u_{2},\nabla u_{2})_g\|_{X_{0}}}\\
 &\le&6\|u_{1}(\nabla u_{1},\nabla u_{1})_g
 -u_{2}(\nabla u_{1},\nabla u_{1})_g\|_{X_0} +6\|u_{2}(\nabla u_{1},\nabla u_{1})_g 
 -u_{2}(\nabla u_{1},\nabla u_{2})_g\|_{X_0} 
 \\
&&   +6\|u_{2}(\nabla u_{1},\nabla u_{2})_g-u_{2}(\nabla u_{2},\nabla u_{2})_g\|_{X_{0}}
\\
&\leq& c_{5}\Big(\|\tilde x^{2(1+\gamma-\frac{n+1}{2})}(u_{1}-u_{2})\|_{\cH^{0,\gg}_p(\B)}
\|u_{1}\|_{X_q}\|u_{2}\|_{X_q}\\
 && +\|\tilde x^{2(1+\gamma-\frac{n+1}{2})}u_{2}\|_{\cH^{0,\gg}_p(\B)}
    (\|u_{1}\|_{X_{q}}+\|u_{2}\|_{X_{q}})\|u_{1}-u_{2}\|_{X_{q}}\Big)\\
 &\le& c_{5}\Big(\|\tilde x^{2(1+\gamma-\frac{n+1}{2})}\|_{\cH^{0,\gg}_p(\B)}\|u_{1}-u_{2}\|_{\infty}
\|u_{1}\|_{X_q}\|u_{2}\|_{X_q}\\
 && +\|\tilde x^{2(1+\gamma-\frac{n+1}{2})}\|_{\cH^{0,\gg}_p(\B)}\|u_{2}\|_{\infty}
    (\|u_{1}\|_{X_{q}}+\|u_{2}\|_{X_{q}})\|u_{1}-u_{2}\|_{X_{q}}\Big)\\
    &\le& c_{6}\Big( \|u_{1}\|_{X_q}\|u_{2}\|_{X_q}+\|u_{2}\|_{X_{q}}
    (\|u_{1}\|_{X_{q}}+\|u_{2}\|_{X_{q}})\Big)\|u_{1}-u_{2}\|_{X_{q}}
\end{eqnarray*}
with suitable constants $c_5,c_6$, and then use the fact 
that $U$ is bounded in $X_q$. 
\end{proof}

\section{The Allen-Cahn Equation}

We will now prove the existence of short time solutions to the Allen-Cahn equation \eqref{AC}, 
with the help of Theorem \ref{CL}. 
We choose $X_{0}=\cH^{0,\gg}_p(\B)$ and  $X_{1}=\cD(\underline \Delta)$
with one of the extensions determined in Theorem \ref{elldomain2}.  
Clearly $c-\underline{\Delta}\in\mathcal{BIP}(\phi)$ for any $\phi>0$, with $c>0$ sufficiently large, and the operator $A=\underline \Delta$ has maximal regularity for $(X_{1},X_{0})$ and any $q$. 
Moreover,  condition (H1) of Theorem \ref{CL} is trivially satisfied.

If we take any open set $U\in X_{q}$ and $u_{1},u_{2}\in U$, then we infer from 
the Lipschitz continuity of $f$ that
\begin{gather*}
\|f(u_{1})-f(u_{2})\|_{X_{0}}\leq c_{f}\|u_{1}-u_{2}\|_{X_{0}}
\leq c_{f}\|u_{1}-u_{2}\|_{X_{q}}
\end{gather*}
for the Lipschitz constant $c_{f}$. Hence,  condition (H2) of Theorem \ref{CL} 
is also satisfied, and we obtain the following result:

\begin{theorem}
Let $\underline{\Delta}$ be a closed extension of the Laplace operator in $\cH^{0,\gg}_p(\B)$ as in Theorem {\rm\ref{elldomain2}}. Then, for any
\begin{gather*}
u_{0}\in X_{q}=(\cH^{0,\gg}_p(\B),\mathcal{D}(\underline{\Delta}))_{1-\frac{1}{q},q},
\end{gather*} 
with $q\in\,]1,\infty[$, there exists a $T>0$ such that the problem \eqref{AC1}, \eqref{AC2} admits a unique solution 
\begin{gather*}
u\in L^q(0,T;\mathcal{D}(\underline{\Delta}))\cap W^{1}_q(0,T; \cH^{0,\gg}_p(\B))
\cap C([0,T],X_q).
\end{gather*}
\end{theorem}

\begin{remark}\rm 
If $\dim\B\geq3$, by restricting the weight to $\gamma\in\ ]-\frac{n+1}{2},\min\{n-3,0\}]$ with $\gamma\neq \frac{n+1}{2}-q_{j}^{\pm}-2$, we can choose the extension $\cD(\underline{\Delta})=\cD(\Delta_{\min})=\mathcal{H}^{2,2+\gamma}_{p}(\mathbb{B})$. 
By Lemma {\rm 5.4} in \cite{CSS1} we then find that
\begin{gather*}
X_{q}=(X_{0},X_{1})_{1-\frac{1}{q},q}
\hookrightarrow(\cH^{0,\gg}_p(\B),\cH^{2,2+\gg}_p(\B))_{1-\frac{1}{q}-\gd,p}
\hookrightarrow \cH^{2(1-\frac{1}{q})-3\gd,2(1-\frac{1}{q})+\gg-3\delta}_p(\B)
\end{gather*} 
for any $\delta>0$. Moreover, for $q=p\le2$, Lemma {\rm 5.4} in \cite{CSS1} even shows that, for
arbitrary  $\delta>0$,
\begin{gather*}
X_{p}=(X_{0},X_{1})_{1-\frac{1}{p},p}
=(\cH^{0,\gg}_p(\B),\cH^{2,2+\gg}_p(\B))_{1-\frac{1}{p},p}
\hookrightarrow \cH^{2(1-\frac{1}{p}),2(1-\frac{1}{p})+\gg-\delta}_p(\B).
\end{gather*} 
\end{remark}

\end{document}